\documentclass[12pt]{article}
\usepackage{geometry}
\usepackage{amsfonts}
\usepackage{dsfont}
\usepackage[latin1]{inputenc}
\usepackage{amssymb , amsmath, amsthm, amsopn, amstext, amscd}
\usepackage[T1]{fontenc}
\usepackage{latexsym, enumerate}
\usepackage{graphics}
\usepackage{graphicx}
\usepackage{multirow}

\DeclareMathOperator*{\argmin}{arg\,min}

\newcommand{\mr}{\mathbb{R}}
\newcommand{\mrpn}{\mathbb{R}_0^+}

\newcommand{\mn}{\mathbb{N}}

\newcommand{\me}{\mathbb{E}}

\newcommand{\mpr}{\mathbb{P}}

\newtheorem{theorem}{Theorem}[section]
\newtheorem{proposition}{Proposition}[]
\newtheorem{lemma}{Lemma}[section]
\newtheorem{corollary}{Corollary}[section]
\newtheorem{remark}{Remark}[section]
\newtheorem{example}{\it $\hookrightarrow$ Example}[section]

\title{How~Correlations~Influence Lasso~Prediction}

\author{Mohamed Hebiri$^{\diamond }$ and Johannes C. Lederer$^{\star }$\footnote{JCL acknowledges partial financial support as member of the
    German-Swiss Research Group FOR916 (Statistical
    Regularization and Qualitative Constraints) with grant number 20PA20E-134495/1.} \\ \\
\small{$^{\diamond }$ Universit\'e Paris-Est -- Marne-la-Vall\'ee,}\\ 
\small{5, boulevard Descartes, Champs sur Marne,}\\ 
\small{77454 Marne-la-Vall\'ee, Cedex 2 France.}\\ \\
\small{$^{\star }$ ETH Z\"urich,}\\
\small{R\"amistrasse, 101}\\
\small{8092 Z\"urich, Switzerland.}
}

\date{}

\begin{document}

\maketitle

\begin{abstract}
We study how correlations in the design matrix influence Lasso
prediction. First, we argue that the higher the correlations are, the
smaller the optimal tuning parameter is. This implies in particular that
the standard tuning parameters, that do not depend on the design matrix,
are not favorable. Furthermore, we argue that Lasso prediction works well
for any degree of correlations if suitable tuning parameters are chosen. We
study these two subjects theoretically as well as with simulations.
\end{abstract}

\section{Introduction}
Although the Lasso estimator is very popular and correlations are present in many of its diverse
applications, the influence of these
correlations is still not entirely understood. Correlations are
surely problematic for parameter estimation and variable selection. The influence
of correlations on prediction, however, is far less
clear.\\

Let us first set the framework for our study. We consider the linear
regression model
\begin{equation}
\label{eq.Correlation.Model}
Y = X\beta_0 + \sigma \epsilon,
\end{equation}
where $Y\in\mr^n$ is the \itshape response vector, \normalfont $X\in\mr^{n\times p}$
is the \itshape design matrix, \normalfont $\epsilon\in\mr^n$ is the
\itshape noise \normalfont and $\sigma\in\mr^+$ is the
\itshape noise level. \normalfont We assume in the following that the noise
level $\sigma$ is known and that the
noise $\epsilon$ obeys an $n$ dimensional normal distribution with
covariance matrix equal to the identity. Moreover, we assume that the
design matrix $X$ is normalized, that is, $\left(X^TX\right)_{jj}=n$ for
$1\leq j \leq p$. Three main tasks are then usually considered: estimating
$\beta_0$ ({\it parameter estimation}), selecting the non-zero
components of $\beta_0$ ({\it
  variable selection}), and estimating $X\beta_0$ ({\it prediction}). Many
applications of the above regression model are {\it high dimensional}, that is, the number of variables $p$ is larger than
the number of observations $n$ but are also {\it sparse}, that is, the true
solution $\beta_0$ has only few nonzero entries. A computationally feasible
method for the mentioned tasks is, for instance, the widely used Lasso estimator introduced in \cite{Tibshirani-LASSO}:
\begin{equation*}
\label{eq.Correlation.Lasso}
  \hat\beta:=\argmin_{\beta\in\mr^p}\left\{\Vert Y-X\beta \Vert_2^2+\lambda\Vert\beta\Vert_1\right\}.
\end{equation*}
In this paper, we focus on the prediction error of this estimator for
different degrees of correlations. 
The literature on the Lasso estimator has become very large, we refer the reader to the well written books  \cite{BookAGS11, Buhlmann11,HastieTibshiraniFriedman2001} and the references therein. \\

Two types of bounds for the prediction error are known in the theory for
the Lasso estimator. On the one hand, there are the so called {\it fast
rate bounds} (see \cite{Bickel09,Lasso2,VandeGeerConditionLasso09}
and references therein). These bounds are nearly optimal but imply
restricted eigenvalues or similar conditions and therefore only apply for
weakly correlated designs. On the other hand, there are the so called {\it
slow rate bounds} (see \cite{HCB08,Kolt10,MM11,RigTsy11}). These bounds are valid for
any degree of correlations but - as their name suggests - are usually
thought of as unfavorable.\\

Regarding the mentioned bounds, one could claim that correlations lead in
general to large prediction errors. However, recent results in
\cite{vdGeer11} suggest that this is not true. It is argued in
\cite{vdGeer11} that for (very) highly correlated designs, small tuning
parameters can be chosen and favorable slow rate bounds are obtained. In the
present paper, we provide more insight into the relation
between Lasso prediction and correlations. We find that the larger the
correlations are, the smaller the optimal tuning parameter is. Moreover, we
find both in theory and simulations that Lasso performs well for any degree
of correlations if the tuning parameter is chosen suitably.\\

We finally give a short outline of this paper: we first discuss
the known bounds on the Lasso prediction error. Then, after some
illustrating numerical results, we study the subject
theoretically. We then present several simulations, interpret our
results and finally close with a discussion.

%%%%%%%%%%%%%%%%%%%%%%%%%%%%%%%%%%%%%%%%%%%%%%%%%%%%%%
%%%%%%%%%%%%%%%%%%%%%%%%%%%%%%%%%%%%%%%%%%%%%%%%%%%%%%
\section{Known Bounds for Lasso Prediction}
\label{sec:Rates}
%%%%%%%%%%%%%%%%%%%%%
%%%%%%%%%%%%%%%%%%%%%
To set the context of our contribution, we first discuss briefly the known
bounds for the prediction error of the Lasso estimator. We refer to the
books \cite{BookAGS11} and \cite{Buhlmann11} for a detailed introduction
to the theory of the Lasso.\\ 

{\it Fast rate bounds}, on the one hand, are bounds proportional to the
square of the tuning parameter $\lambda$. These bounds are only valid for 
weakly correlated design matrices. We first recall the corresponding
assumption. Let $a$ be a vector in $\mathbb{R}^p$, $J$ a subset of $\{1,\ldots,p\}$, and finally $a_J$ the
vector in $\mr^p$ that has the same coordinates as $a$ on $J$ and zero
coordinates on the complement $J^c$. Denote the cardinality of a given set
by $|\cdot|$. For a given integer $\bar s$, the {\it Restricted Eigenvalues (RE)} assumption introduced in~\cite{Bickel09}
reads then
\begin{description}
	\item[{\bf Assumption~RE($ \bar s $)}:]
	\begin{equation*}
	\label{eq:CondHypEase}
		\phi(\bar s ) := \min_{J_0 \subset \{1,\ldots,p \}: |J_0| \leq  \bar s } \min_{\Delta \neq 0: \Vert \Delta_{J_0^c} \Vert_1 \leq 3 \Vert \Delta_{J_0} \Vert_1 } \frac{ \Vert X \Delta \Vert_2}{ \sqrt{n} \Vert \Delta_{J_0} \Vert_2 }  
		>0.
	\end{equation*}
\end{description}
The integer $\bar s$ plays the role of a sparsity index and is usually 
comparable to the number of nonzero entries of $\beta_0$.
More precisely, to obtain the following fast rates, it is assumed that $\bar s \geq s$, where $s:=\vert \{j: (\beta_0)_j \neq 0\}\vert $.
Also, we notice that $\phi(\bar
s)\approx 0$ corresponds to correlations. Under the above assumption it holds (see for example Bickel et
al.~\cite{Bickel09} and more recently Koltchinskii et al.~\cite{Kolt10}):
\begin{equation}
\label{eq:boundLassoL0}
\Vert
X(\hat\beta-\beta_0)\Vert_2^2 \leq \frac{\lambda^2\bar s}{n \phi^2(\bar{s})}
\end{equation} 
on the set $\mathcal{T}:=\left\{\sup_{\beta}\frac{2\sigma|\epsilon^TX\beta|}{\Vert
     \beta\Vert_1}\leq \lambda \right\}$. Similar bounds, under slightly
 different assumptions, can be found in \cite{VandeGeerConditionLasso09}.
Usually, the tuning parameter $\lambda$ is chosen proportional to $\sigma
\sqrt{n \log(p)}$. For fixed $\phi$, the above rate then is optimal up to a logarithmic term
(see~\cite[Theorem~5.1]{BTWAggSOI}) and the set $\mathcal{T}$ has a high
probability (see Section~\ref{sec.ad}). For correlated designs, however, this choice of the tuning parameter is
not suitable. This is detailed in the following section.\\

{\it Slow rate bounds}, on the other hand, are bounds only proportional to
the tuning parameter $\lambda$. These bounds are valid for arbitrary
designs, in particular, they are valid for highly correlated designs. The result \cite[Eq. (2.3) in Theorem 1]{Kolt10}
yields in our setting
 \begin{equation}
 \label{eq.slowRate}
   \Vert X(\hat\beta-\beta_0)\Vert_2^2\leq 2\lambda\Vert \beta_0\Vert_1
 \end{equation}
\noindent on the set $\mathcal{T}$. Similar bounds can be found in \cite{HCB08} (for a related work on a truncated version of the Lasso), in \cite[Theorem 3.1]{MM11} (for estimation of a general function in a Banach space), and in \cite[Theorem~4.1]{RigTsy11} (which also applies to the non-parametric setting; note that the corresponding bound can be written in the form above with arbitrary tuning parameter $\lambda$). 
	We note that these bounds depend on  $\Vert\beta_0\Vert_1$ instead
        of $\bar s$. Moreover, they depend on $\lambda$ to the
first power, and these bounds
are therefore considered unfavorable compared to the fast rate bounds.\\

The mentioned bounds are only useful for sufficiently large tuning parameters
such that the set $\mathcal{T}$ has a high probability. This is crucial
for the following. We show that the higher the correlations, the larger the
probability of $\mathcal{T}$ is. Correlations thus allow for small tuning
parameters; this implies for correlated designs, via the factor $\lambda$
in the slow rate bounds, favorable bounds even though no fast rate
bounds are available.

\begin{remark}
\label{rk.slowRateImproved}
The slow rate bound \eqref{eq.slowRate} can be improved if  $\Vert \beta_0\Vert_1$ is large. Indeed, we proof in the Appendix that
 \begin{equation*}
 \label{eq.slowRateImproved}
   \Vert X(\hat\beta-\beta_0)\Vert_2^2\leq 2  \lambda  \min \left\{ \Vert \beta_0\Vert_1\ , \  \Vert (\hat\beta-\beta_0)_{J_0}  \Vert_1 \right\}
 \end{equation*}
\noindent on the set $\mathcal{T}$. That is, the prediction error can be
bounded both with $\Vert \beta_0\Vert_1$ and with the $\ell_1$
estimation error restricted on the sparsity pattern of $\beta_0$. The latter term
can be considerably smaller than $\Vert \beta_0\Vert_1$ (in particular for
weakly correlated designs). A detailed analysis of this observation, however, is not within the scope of the
present paper.
\end{remark}

%%%%%%%%%%%%%%%%%%%%%%%%%%%%%%%%%%%%%%%%%%%%%%%%%%%%%%%%
%%%%%%%%%%%%%%%%%%%%%%%%%%%%%%%%%%%%%%%%%%%%%%%%%%%%%%%%

\section{The Lasso and Correlations}
\label{secFirstEvid}
We show in this section that correlations strongly influence the optimal
tuning parameters. Moreover, we show that - for suitably chosen tuning
parameters - Lasso performs well in prediction for different levels
of correlations. For this, we first present simulations where we compare
Lasso prediction for an initial design with Lasso prediction for an
expanded design with additional impertinent variables. Then, we discuss the
theoretical aspects of correlations. We introduce, in particular, a simple
and illustrating notion about correlations. Further simulations finally
confirm our analysis.

\subsection{The Lasso on Expanded Design Matrices}
Is Lasso prediction becoming worse when many impertinent variables are added to the
design? Regarding the bounds and the usual value of the tuning parameter $\lambda$ described in the last section, one may expect
that many additional variables lead to notably larger optimal tuning parameters and prediction
errors. However, as we see in the following, this is not true in general.\\ 

Let us first describe the experiments. 

\paragraph{Algorithm 1}
We simulate from the linear regression model~\eqref{eq.Correlation.Model}
and take as input the number of observations $n$, the number of variables
$p$, the noise level $\sigma$, the number of nonzero entries of the true
solution $s := \left\{ j: (\beta_0)_j \neq 0 \right\}$ and finally a
correlation factor $\rho\in[0,1)$. We then sample the $n$ independent rows of the design
matrix $X$ from a normal distribution with mean zero and covariance matrix with
diagonal entries equal to $1$ and off-diagonal entries equal to
$\rho$, and we normalize $X$ such that $(X^{\top}X )_{jj} = n$ for $1\leq j \leq n$. Then, we define $(\beta_0)_i:=1$ for $1\leq i \leq s$ and
$(\beta_0)_i:=0$ otherwise, sample the error $\epsilon$ from a standard
normal distribution and compute the response vector $Y$ according to
\eqref{eq.Correlation.Model}. After calculating the Lasso solution $\hat \beta$, we finally compute
the prediction error $\Vert X(\hat\beta-\beta_0)\Vert_2^2$ for different
tuning parameters $\lambda$ and find the optimal tuning parameter, that is,
the tuning parameter that leads to the smallest prediction error.

\paragraph{Algorithm 2}
This algorithm only differs from the above algorithm in one point. In an
additional step after the initial design matrix $X$ is sampled, we add for
each column $X^{(j)}$ of the initial design matrix $p\text{~-~}1$ columns
sampled according to $X^{(j)} + \eta N$. We finally normalize the resulting matrix.
The parameter $\eta$ controls the
correlation among the added columns and the initial columns and $N$ is a
standard normally distributed random vector. Compared to the initial design, we have
now a design with $p^2\text{~-~}p$ additional impertinent variables.
\\

Several algorithms for computing a Lasso solution have been proposed: For example, using interior point methods \cite{CDS98}, using homotopy parameters \cite{Efron-LARS,O02, T05}, or using a so-called shooting algorithm \cite{DDM04,F98,FHHT07}. We use the LARS algorithm introduced
in~\cite{Efron-LARS}, since, among others, Bach et
al.~\cite[Section~1.7.1]{BJMO11} have confirmed the good behavior of this
algorithm when the variables are correlated.

 \begin{figure}[ht]
\begin{center}
\vskip -0.1in
\includegraphics[width=5in] {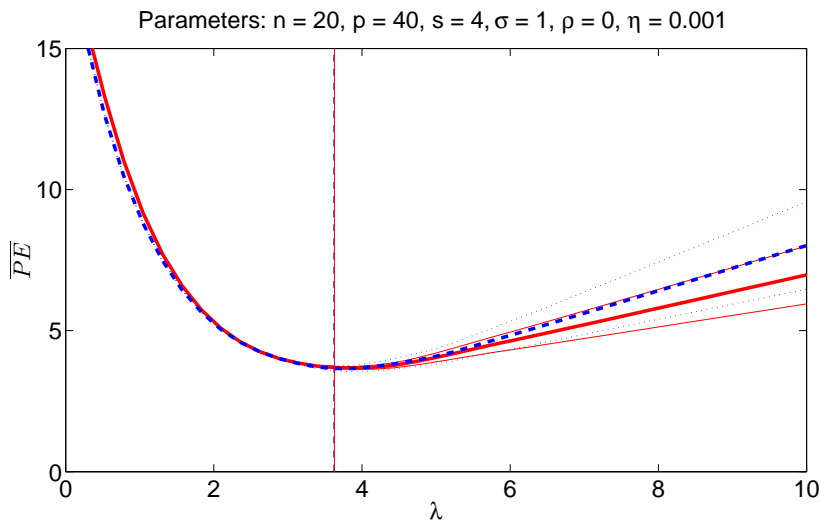}
\vglue-15pt
~
\begin{minipage}[t]{0.90\textwidth}
\caption{\label{fig:FigCom1}\footnotesize We plot the mean values
  $\overline{PE}$ of the prediction errors $\|A\hat\beta - A\beta_0
  \|_2^2$ for $1000$ iterations as a function of the tuning parameter
  $\lambda$. The blue, dashed line corresponds to Algorithm 1, where $A$
  stands for the
  initial design matrices. The blue, dotted lines give the
  confidence bounds. The red, solid line corresponds to Algorithm 2, where
  $A$ represents the extended matrices. The faint, red lines give the confidence
  bounds. The parameters for the algorithms are given in the header. The mean of the optimal tuning
  parameters is $3.62\pm 0.01$ for Algorithm 1 and $3.63\pm
  0.01$ for Algorithm 2. These values are represented by the blue and red vertical lines.}
\end{minipage}
\end{center}
%\vspace{-0.27in}
\end{figure}

\begin{figure}[ht]
\begin{center}
\vskip -0.1in
\includegraphics[width=5in] {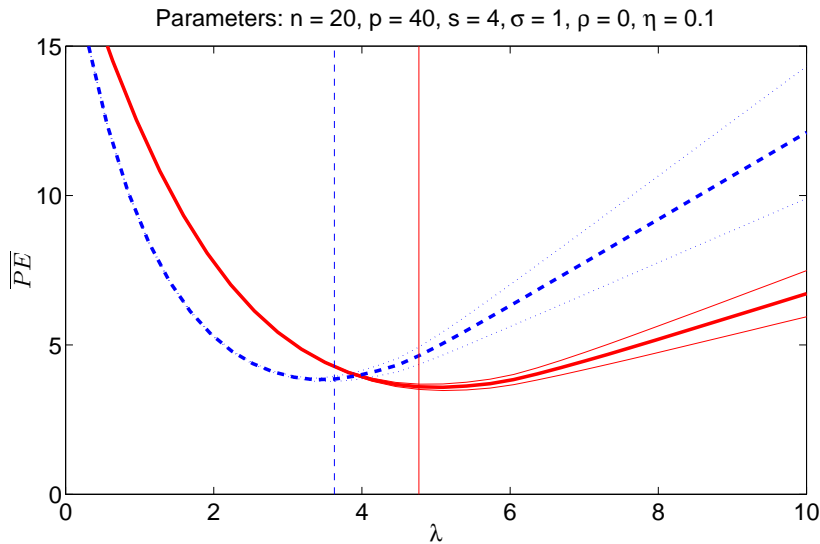}
\vglue-5pt
~
\begin{minipage}[h]{0.90\textwidth}
\caption{\label{fig:FigCom2}\footnotesize We plot the mean values
  $\overline{PE}$ of the prediction errors $\|A\hat\beta - A\beta_0
  \|_2^2$ for $1000$ iterations as a function of the tuning parameter
  $\lambda$. The blue, dashed line corresponds to Algorithm 1, where A
  stands for the
  initial design matrices. The blue, dotted lines give the
  confidence bounds. The red, solid line corresponds to Algorithm 2, where
  A stands for the extended matrices. The faint, red lines give the confidence
  bounds. The parameters for the algorithms are given in the header. The mean of the optimal tuning
  parameters is $3.63\pm 0.01$ for Algorithm 1 and $4.77\pm
  0.01$ for Algorithm 2. These values are represented by the blue and red vertical
  lines.}
\end{minipage}
\end{center}
%\vspace{-0.27in}
\end{figure}

\paragraph{Results}
We did $1000$ iterations of the above algorithms for different $\lambda$ and with $n=20$,
$p=40$, $s=4$, $\sigma = 1$, $\rho = 0$ and with $\eta =
0.001$ (Figure \ref{fig:FigCom1}) and $\eta = 0.1$ (Figure
\ref{fig:FigCom2}). We plot the means $\overline{PE}$ of the
prediction errors as a function of $\lambda$. The blue, dashed curves correspond
to the initial designs (Algorithm 1), the red, solid curves correspond to
the extended designs (Algorithm 2). The confidence bounds are
plotted with faint lines in the according color and finally the mean values of the optimal
tuning parameters are plotted with vertical lines.\\
We find in both examples that the minimal prediction errors,
that is, the minima of the red and blue curves, do not differ
significantly. Additionally, in the first example, corresponding to highly correlated added variables ($\eta =
0.001$, see Figure \ref{fig:FigCom1}), also the optimal tuning parameters do
not differ significantly. However, in the second example ($\eta =
0.1$, see Figure \ref{fig:FigCom2}), the optimal tuning parameter is considerably larger for
the extended designs.

\paragraph{First Conclusions}
Our results indicate that tuning parameters proportional to $\sqrt{n\log
p}$ (cf. \cite{Bickel09} and most other contributions on the subject)
independent of the degree of correlations are not favorable.  Indeed, for
Algorithm 2, this would lead to a tuning parameter proportional to
$\sqrt{n\log p^2}=\sqrt{2n\log p}$, whereas for Algorithm~1 to $\sqrt{n\log
p}$. But regarding Figure~\ref{fig:FigCom1}, the two optimal tuning parameters
are nearly equal and hence these choices are not favorable. In contrast,
the results illustrate that the optimal
tuning parameters depend strongly on the level of correlations: for
Algorithm~2 (red, solid curves), the means of the optimal tuning parameters corresponding
to the highly correlated case ($3.63\pm 0.01$, see Figure~\ref{fig:FigCom1}) are be considerably
smaller than the ones corresponding to the weakly
correlated case ($4.77\pm 0.01$, see Figure~\ref{fig:FigCom2}).\\ 
Our results indicate additionally that the minimal mean prediction errors
are comparable for all cases. This implies, that a suitable tuning
parameters lead to good prediction even with additional impertinent
parameters. We only give two examples here but made these observations for
any values of $n$, $p$ and $s$.

%%%%%%%%%%%%%%%%%%%%%%%%%%%%%%%%%%%%%%%%%%%%%%%%%%%%%%%%%%%
%%%%%%%%%%%%%%%%%%%%%%%%%%%%%%%%%%%%%%%%%%%%%%%%%%%%%%%%%%%

\subsection{Theoretical Evidence}
We provide in this section theoretical explanations for the above
observations. For this, we first discuss results derived in
\cite{vdGeer11}. They find that high correlations allow for small tuning
parameters and that this can lead to bounds for Lasso prediction that are
even more favorable than the fast rate bounds. Then, we introduce and apply
new correlation measures that provide some insight for (in contrast to
\cite{vdGeer11}) arbitrary degrees of correlations. For no correlations,
in particular, these results simplify to the classical results.

%%%%%%%%%%%%%%%%%%%%%%%%%
%%%%%%%%%%%%%%%%%%%%%%%%%
\subsubsection{Highly Correlated Designs}
\label{sec:SY}
%%%%%%%%%%%%%        %%%%%%%%%%%%
%%%%%%%%%%%%%%%%%%%%%%%%%
First results for the highly correlated case are derived in
\cite{vdGeer11}. Crucial in their study is the treatment of the
stochastic term with metric entropy.\\

The bound on the prediction error for Lasso reads as follows:
\begin{lemma}\cite[Theorem~4.1 \& Corollary~4.2]{vdGeer11}\label{lemma:SYbasic}
 On the set 
  \begin{equation*}
    \mathcal T_\alpha:=\left\{\sup_{\beta}\frac{2\sigma|\epsilon^TX\beta|}{\Vert
      X\beta \Vert_2^{1-\alpha}\Vert \beta\Vert_1^\alpha}\leq \widetilde{\lambda} \right\}
  \end{equation*}
  we have for $\lambda=(2\widetilde{\lambda}n^{\alpha-1})^\frac{2}{1+\alpha}\Vert
  \beta_0\Vert_1^{\frac{\alpha-1}{1+\alpha}}$ and $0< \alpha < 1$
  \begin{equation*}
    \Vert X(\hat\beta-\beta_0)\Vert_2^2\leq
   \frac{21}{2}\left(2\widetilde{\lambda} n^{\alpha-1}\right)^\frac{2}{1+\alpha}\Vert \beta_0\Vert_1^\frac{2\alpha}{1+\alpha}.
  \end{equation*}
\end{lemma}
\noindent We show in the following that for high correlations the
stochastic term $\mathcal T_\alpha$ has a high probability even for small
$\alpha$ and thus favorable bounds are obtained. The parameter $\alpha$ can
be thought of as a measure of correlations: $\alpha$ small corresponds to
high correlations, $\alpha$ large corresponds to small correlations.\\

The stochastic term $\mathcal T_\alpha$ is estimated using metric
entropy. We recall, that the \itshape covering numbers \normalfont
$N(\delta, \mathcal F, d)$ measure the complexity of a set $\mathcal F$
with respect to a metric $d$ and a radius $\delta$. Precisely, $N(\delta,
\mathcal F, d)$ is the minimal number of balls of radius $\delta$ with
respect to the metric $d$ needed to cover $\mathcal F$. The \itshape
entropy numbers \normalfont are then defined as $H(\delta, \mathcal F,
d):=\log N(\delta, \mathcal F, d)$. In this framework, we say that the
design is highly correlated if the covering numbers
$N(\delta,\operatorname{sconv}\{X^{(1)},...,X^{(p)}\},\Vert \cdot \Vert_2)$
(or the corresponding entropy numbers) increase only mildly with
$1/\delta$, where $\{X^{(1)},...,X^{(p)}\}$ are the columns of the design
matrix and $\operatorname{sconv}$ denotes the symmetric convex hull. This is specified in the following lemma:

\begin{lemma}\cite[Corollary~5.2]{vdGeer11}\label{lemma:SJ1} Let $0<
  \alpha <1 $ be fixed. Then, assuming
  \begin{equation}\label{eq:rent}
 \log\left(1+N(\sqrt n\delta,\operatorname{sconv}\{X^{(1)},...,X^{(p)}\},\Vert \cdot \Vert_2)\right)\leq
 \left(\frac{A}{\delta}\right)^{2\alpha},0<\delta\leq 1,
  \end{equation}
  there exists a value $C(\alpha,A)$ depending on
  $\alpha$ and $A$ only such that for all $\kappa>0$ and for
  \begin{equation*}
 \widetilde{\lambda}= \sigma C(\alpha,A)\sqrt{ n^{2-\alpha} \log(2/\kappa)},
  \end{equation*}
the following bound is valid:
\begin{equation*}
  \mpr(\mathcal T_\alpha)\geq 1-\kappa.
\end{equation*}
\end{lemma}
\noindent We observe indeed that the smaller $\alpha$ is, the higher the 
correlations are. We also mention that Assumption~\eqref{eq:rent} only applies to highly
correlated designs. An example is given in \cite{vdGeer11}: the assumption
is met if the eigenvalues of the Gram matrix $\frac{X^TX}{n}$
decrease sufficiently fast.\\

Lemma~\ref{lemma:SYbasic} and Lemma~\ref{lemma:SJ1} can now be combined to
the following bound for the prediction error of the Lasso:
\begin{theorem}\label{thm:SYBase}
  With the choice of $\lambda$ as in Lemma~\ref{lemma:SYbasic} and under the assumptions of Lemma~\ref{lemma:SJ1}, it holds that
\begin{equation*}
    \Vert X(\hat\beta-\beta_0)\Vert_2^2\leq
   \frac{21}{2}\left(\sigma C(\alpha,A)\sqrt{ n^{\alpha}\log(2/\kappa)}\right)^\frac{2}{1+\alpha}\Vert \beta_0\Vert_1^\frac{2\alpha}{1+\alpha}
  \end{equation*}
 with probability at least $1-\kappa$.
\end{theorem}
\noindent High correlations allow for small values of $\alpha$ and lead
therefore to favorable bounds. For $\alpha$ sufficiently small, these
bounds may even outmatch the classical fast rate bounds. For moderate or
weak correlations, however, Assumption~\eqref{eq:rent}
is not met and therefore the above lemma does not apply.

%!TEX root = LaClasse.tex
%%% Local Variables: 
%%% mode: pdflatex
%%% TeX-master: "LaClasse"
%%% End: 
%%%%%%%%%%%%%%%%%%%%%%%%%%%%%
%%%%%%%%%%%%%%%%%%%%%%%%%%%%%
\subsubsection{Arbitrary Designs}
\label{sec.ad}
%%%%%%%%%%%%%%%%%%%%%%%%%%%%%
%%%%%%%%%%%%%%%%%%%%%%%%%%%%%
In this section, we introduce bounds that apply to any degree of
correlations. The correlations are in particular allowed to be moderate or
small and the bounds simplify to the classical results for weakly correlated
designs. We first introduce two measure for correlations that are then used to
bound the stochastic term. These results are then combined with the
classical slow rate bound to obtain a new bound for Lasso prediction. We
finally give some simple examples.\\

%%%%%%%%%%%%%%%%%%%%%%%%%%%%%
%%%%%%%%%%%%%%%%%%%%%%%%%%%%%
\paragraph{Two Measures for Correlations}
%%%%%%%%%%%%%%%%%%%%%%%%%%%%%
%%%%%%%%%%%%%%%%%%%%%%%%%%%%%
We introduce two numbers that measure the correlations in the design. For
this, we first define the \itshape correlation function \normalfont
$K:\mr^+_0\to \mn$ as
\begin{align}\label{eq:Kfunc}
K(x):=\min\{&l\in\mn: \exists x^{(1)},\ldots,x^{(l)}\in \sqrt
nS^{n-1},\\&X^{(m)}\in (1+x)\operatorname{sconv}\{x^{(1)},\ldots,x^{(l)}\}~\forall 1\leq m\leq p\},\nonumber
\end{align}

%\begin{figure*}[htb!]
%\begin{minipage}[b]{0.99\linewidth}
%\centerline{
 % \includegraphics[width=5cm]{3b}}
%\end{minipage}
%\caption{Some caption}
%\label{fig:Correlation}
%\end{figure*}

%\begin{figure}\centering
%\includegraphics[clip=true,scale=0.3]{1b.eps}
%\includegraphics[clip=true,scale=0.5]{3b.eps}
%\caption{Momo: Do you know how to make nice pictures of this kind??? Yoyo: non, mais je vois que toi tu te la racontes!!! En fit, je crois que je peux trouver comment le faire si tu ne te rappelles pas ;)}
%\label{fig:Correlation}
%\end{figure}
\noindent where $S^{n-1}$ denotes the unit sphere in $\mr^n$. We observe that $K(x)\leq p$ for all $x\in\mrpn$ and that $K$ is
a decreasing function of $x$. A measure for correlations should, as the
metric entropy above, measure how close to one another the columns of the
design matrix are. Indeed, for moderate $x$, $K(x)\approx p$ for uncorrelated designs, whereas $K(x)$
may be considerably smaller for correlated designs.  This information is concentrated in the
\itshape correlation factors 
\normalfont that we define as
\begin{equation*}
K_\kappa:=\inf_{x\in\mrpn}(1+x)\sqrt{\frac{\log(2K(x)/\kappa)}{\log(2p/\kappa)}},
\end{equation*}
$\kappa\in(0,1]$, and as
\begin{equation*}
F:=\inf_{x\in\mrpn}(1+x)\sqrt{\frac{\log(1+K(x))}{\log(1+p)}}.
\end{equation*}
\noindent Since $K(0)\leq p$, it holds that $F,K_\kappa\in
(0,1]$. We also note that similar quantities could can be defined for 
$p=\infty$ (removing the normalization). In any case, large $F$ and $K_\kappa$
correspond to uncorrelated designs, whereas small $F$ and $K_\kappa$ correspond to
correlated designs.

%%%%%%%%%%%%%%%%%%%%%%%%%%%%%
%%%%%%%%%%%%%%%%%%%%%%%%%%%%%
\paragraph{Control of the Stochastic Term}
%%%%%%%%%%%%%%%%%%%%%%%%%%%%%
%%%%%%%%%%%%%%%%%%%%%%%%%%%%%
We now show that small correlation factors allow for small tuning
parameters and thus lead to favorable bounds for Lasso prediction. Crucial in our analysis is again the treatment of a stochastic term similar to the one above.\\

%
%
%W
We prove the following bound in the appendix:
\begin{theorem}\label{lemma.Correlation.main} With the definitions above,
  $\mathcal T$ as defined in Section~\ref{sec:Rates} and for all $\kappa>0$
  and $\lambda \geq
  \lambda_\kappa:=K_\kappa2\sigma\sqrt{2n\log(2p/\kappa)}$, it holds that
\begin{equation*}\label{eq.Correlation.main2}
\mpr\left(\mathcal{T}\right) \geq 1-\kappa.
\end{equation*}
Additionally, independently of the choice of $\lambda$,  %
\begin{equation*}\label{eq.Correlation.main1}
\me\left[\sup_{\Vert\beta\Vert_1\leq M}\sigma\mid\epsilon^T X\beta\mid\right]\leq F\sigma\sqrt{\frac{8n\log(1+p)}{3}}M.
\end{equation*}
\end{theorem}
\noindent For small correlation factors $K_\kappa$ and $F$,  the
minimal tuning parameters $\lambda_\kappa$ and the expectation of the
stochastic term are small. For $K_\kappa\to 1$, the minimal tuning parameters
simplify to $2\sigma\sqrt{2n\log(2p/\kappa)}$ (cf. \cite{Bickel09}). Similarly, for $F\to 1$, the expectation of the
stochastic term simplifies to $\sigma
\sqrt{\frac{8n\log(1+p)}{3}}M$.\\

Together with the slow rate bound introduced in Section~\ref{sec:Rates},
this permits 
the following bound:
\begin{corollary}\label{thm.YNew}
  For $\lambda\geq\lambda_\kappa$ it holds that
 \begin{equation*}
   \Vert X(\hat\beta-\beta_0)\Vert_2^2\leq 2\lambda\Vert \beta_0\Vert_1
 \end{equation*}
with probability at least $1-\kappa$. 
\end{corollary}
\noindent Our contribution to this result concerns the tuning
parameters:  for the classical value
$\lambda=2\sigma\sqrt{2n\log(2p/\kappa)}$, the bound simplifies to the
classical slow rate bounds. Correlations, however, allow for smaller $\lambda$ and thus
lead to more favorable bounds.\\

Let us have a look at some general aspects of the above bound. Most importantly,
Corollary~\ref{thm.YNew} applies to any degree of correlations. In contrast,
Theorem~\ref{thm:SYBase} only applies for highly correlated
designs, and the classical fast rate bounds only apply for weakly correlated
designs. We also observe that the sparsity index $\bar s$, which appears in the
classical fast rate bounds, does not appear in
Corollary~\ref{thm.YNew}. Hence, 
Corollary~\ref{thm.YNew} can be useful even if the true regression vector $\beta_0$ is only
approximately sparse. However, for very large $\|\beta_0\|_1$, the bound is
unfavorable. An example in this context can be found in \cite[Section
2.2]{Candes09}. (However, we do not agree with their conclusions corresponding to this example. 
We stress, in contrast, that 
correlation are not problematic in general.) We finally refer to
Remark~\ref{rk.slowRateImproved} and to  \cite{MM11} for some additional 
considerations on the optimality of this kind of bounds.

\begin{remark}[Weakly Correlated Designs]
Corollary~\ref{thm.YNew} holds for any degree of correlations because it
follows directly from the classical slow rate bound and the properties of
the refined tuning parameter. However, for weakly correlated designs, we
can also make use of the refined tuning parameter to improve the classical
fast rate bound by a factor $K_\kappa^2$. Usually, the gain is moderate,
but in special cases such as sparse designs
(see Example~\ref{example.SparseDesign}), the gain can be large. Finally,
we note that optimal rates, that is, lower bounds for the prediction error $\Vert
X(\hat\beta-\beta_0)\Vert_2^2$, are available for weakly correlated
designs. The optimal rate
is then $s\log(1+\frac{p}{s})$ and is deduced from Fano's Lemma (see
\cite[Theorem 5.1, Lemma A.1]{BTWAggSOI}). 
For similar results in matrix
regression, we refer to
\cite{Kolt10}, where the assumptions on the design needed for such
lower bounds are explicitly stated (see, in
particular, their Assumption~2).
  \end{remark}

\paragraph{Examples}
In this final section, we illustrate some properties of the
correlation numbers $K_\kappa$ and $F$ for various settings. We consider, in particular, design matrices with different geometric properties and with different degrees of correlations. 

\vspace{1.5mm}

\begin{example}[Low Dimensional Design]\label{example.Correlation.exLowDim} 
Let $\operatorname{dim}~\operatorname{span}\{X^{(1)},...,X^{(p)}\}\leq W$. Then,
\begin{equation*}
X^{(j)}\in \sqrt W ~\operatorname{sconv}\{x_1,...,x_W\}~\text{for all}~ 1\leq j \leq p
\end{equation*}
for properly chosen $x_1,...,x_W\in \sqrt n S^{n-1}$ (for example orthogonal vectors in a suitable subspace). Hence,
$K_\kappa\leq(1+\sqrt W)\sqrt{\frac{\log(2W/\kappa)}{\log(2p/\kappa)}}$ and $F\leq (1+\sqrt
W)\sqrt{\frac{\log(1+W)}{\log(1+p)}}$. 
\end{example}

\vspace{1.5mm}

\begin{example}[Sparse Design]\label{example.SparseDesign} 
Let the number of non-zero coefficients in $X^{(j)}$ be such that $\Vert X^{(j)}\Vert_0\leq d$ for all $1\leq j \leq p$, with $d\leq n$. Then,
\begin{equation*}
X^{(j)}\in \sqrt d\operatorname{sconv}\{x_1,...,x_n\}~\text{for all}~ 1\leq j \leq p
\end{equation*}
for properly chosen $x_1,...,x_n\in\sqrt n S^{n-1}$ (for example orthogonal vectors in $\mr^n$). Hence,
$K_\kappa\leq\sqrt d\sqrt{\frac{\log(2n/\kappa)}{\log(2p/\kappa)}}$ and $F\leq \sqrt d\sqrt{\frac{\log(1+n)}{\log(1+p)}}$.
\end{example}

\vspace{1.5mm}

\begin{example}[Weakly Correlated Design]\label{example.Correlation.vhdd}
  We consider a weakly correlated design with $p\gg n$. It turns out
  that the results in Section~\ref{sec:SY} lead to a
  tuning parameter $\lambda\sim n$ . This implies in particular
   that the results in Section~~\ref{sec:SY}, unlike the results in this section, do not simplify to the
  classical results for uncorrelated designs. We give a sketch of the proof in the Appendix.
\end{example}

\vspace{1.5mm}

\begin{example}[Highly Correlated Design]\label{ex.asym} 
We illustrate with an example that highly correlated designs indeed involve small correlation factors $K_\kappa$. This implies, in accordance to
the simulation results, that tuning parameters depending on the design
rather than tuning parameters depending only on $\sigma$, $n$, and $p$ should be
applied.\\
Let us describe how the design matrix $X$ is chosen: Fix an arbitrary $n$ dimensional vector $X^{(1)}$ of Euclidian norm $\sqrt{n}$. 
Then, we add $p-1$ columns
sampled according to $X^{(1)} + \nu N$ where $\nu$ is a positive (but small) constant and $N$ is a
standard normally distributed random vector.
We finally normalize the resulting matrix.
 \\
For simplicity, we do not give explicit constants but present a coarse 
asymptotic result that highlights the main ideas. For this, we consider the
number of variables $p$ as a
function of the number of observations $n$ and impose the usual
restrictions in high dimensional regression, that is, $p \sim n^w$ for a
$w\in\mn$. If the correlations
are sufficiently large, more precisely $\nu\leq\frac{1}{\sqrt 8 n}$, it
holds that
\begin{equation*} 
K_\kappa \to  \frac{1}{\sqrt{w}}\ \text{~as~}n\to\infty
  \end{equation*} 
with probability tending to 1. This implies especially that the standard tuning
parameters $\sim \sigma\sqrt{n\log p}$ can be considerably too large 
if $\nu$ is large. The proof of the above statement is given in the Appendix.
\end{example}

\vspace{1.5mm}

\begin{example}[Equal Columns]
Let the cardinality of the set $|\{ X^{(j)} :1\leq j \leq p\}|=v$.
Then,
$K_\kappa\leq\sqrt{\frac{\log(2 v/\kappa)}{\log(2p/\kappa)}}$ and
$F\leq
\sqrt{\frac{\log(1+v)}{\log(1+p)}}$.
\end{example}

%%%%%%%%%%%%%%%%%%%%%%%%%%%%%%%%%%%%%%%%%%%%%%%%%%%%%%%%%%%%%%%%%
%%%%%%%%%%%%%%%%%%%%%%%%%%%%%%%%%%%%%%%%%%%%%%%%%%%%%%%%%%%%%%%%%

\subsection{Experimental Study}
\label{sec:Sim}

We consider Algorithm~1 with different sets of parameters to make
statements about the influence of the single parameters on Lasso prediction. In particular,
we are interested in the influence of the correlations $\rho$.\\

% uses multirow
\begin{table}[ht]
\caption{The means of the optimal tuning parameters
  $\overline{\lambda}_{min}$  and the means of the minimal prediction errors
  $\overline{PE}_{min}$ calculated according to Algorithm~1 with 1000
  iterations and for different sets of parameters.}
\label{table.comparisioncorrelation}
\begin{center}
\begin{sc}
	\begin{tabular}{|l|c|c|c|c||c|c|}
		\hline
		$n$& $p$& $s$& $\sigma$& $\rho$& $\overline{\lambda}_{min}$& $\overline{PE}_{min}$ \\
		\hline
		\hline
		 \multirow{3}{*}{$20$}&  \multirow{3}{*}{$40$}&\multirow{3}{*}{$4$}&\multirow{3}{*}{$1$}&{$0.99$} &$0.69\pm 0.03$ &$1.77\pm0.05$\\	\cline{5-7}
		 & & & & {$0.9$} & $1.58\pm 0.03$&$2.37\pm 0.04$\\ 
	\cline{5-7}
		 & & & & {$0$} & $3.60\pm 0.03$&$3.17\pm 0.03$\\ \hline
		 \hline
		 \multirow{3}{*}{$50$}&  \multirow{3}{*}{$40$}&\multirow{3}{*}{$4$}&\multirow{3}{*}{$1$}&{$0.99$} &$0.67\pm0.03$ &$1.29\pm0.04$\\	\cline{5-7}
		 & & & & {$0.9$}&$1.71\pm0.03$ &$1.85\pm0.03$\\ 
	\cline{5-7}
		 & & & & {$0$} &$3.91\pm0.03$ &$3.48\pm0.02$\\ \hline
		 \hline
		 \multirow{3}{*}{$20$}&  \multirow{3}{*}{$400$}&\multirow{3}{*}{$4$}&\multirow{3}{*}{$1$}&{$0.99$} & $0.97\pm0.03$&$1.75\pm0.05$\\	\cline{5-7}
		 & & & & {$0.9$} &$2.11\pm0.04$&$2.58\pm0.04$\\ 
	\cline{5-7}
		 & & & & {$0$}& $4.82\pm0.03$& $3.34\pm0.03$\\ \hline
		 \hline
		 \multirow{3}{*}{$20$}&  \multirow{3}{*}{$40$}&\multirow{3}{*}{$10$}&\multirow{3}{*}{$1$}&{$0.99$} &$0.59\pm0.03$ &$6.50\pm0.22$\\ 	\cline{5-7}
		 & & & & {$0.9$} & $1.46\pm0.03$&$7.97\pm0.19$\\ 
	\cline{5-7}
		 & & & & {$0$} &$2.90\pm0.03$ &$6.65\pm0.06$\\ \hline
		 \hline
		 \multirow{3}{*}{$20$}&  \multirow{3}{*}{$40$}&\multirow{3}{*}{$4$}&\multirow{3}{*}{$3$}&{$0.99$} &$2.42\pm0.15$ &$6.16\pm0.17$\\	\cline{5-7}
		 & & & & {$0.9$} &$5.33\pm0.13$ &$6.47\pm0.14$\\ 
	\cline{5-7}
		 & & & & {$0$} &$12.33\pm0.10$ &$3.80\pm0.03$\\ \hline
	\end{tabular}
\end{sc}
\end{center}
\end{table}

%%%%%%%%%%%%%%%%%%%%%%%%%%%%%%%% Debut Comment %%%%%%%%%%%%%%%%%%%%%%
% 		 \begin{comment}\hline
% 		 \multirow{3}{*}{$20$}&  \multirow{3}{*}{$40$}&\multirow{3}{*}{$4$}&\multirow{3}{*}{$0.1$}&{$0.99$} &$0.05\pm0.01$ &$0.87\pm0.03$\\	\cline{5-7}
% 		 & & & & {$0.9$} &$0.14\pm0.01$ &$0.90\pm0.03$\\ 
% 	\cline{5-7}
% 		 & & & & {$1$} &$0.14\pm0.01$ &$0.32\pm0.01$\\  \hline\hline
% 		 \multirow{3}{*}{$20$}&  \multirow{3}{*}{$40$}&\multirow{3}{*}{$4$}&\multirow{3}{*}{$1.5$}&{$0.99$} &$1.20\pm0.06$ &$2.89\pm0.09$\\	\cline{5-7}
% 		 & & & & {$0.9$} &$2.55\pm0.06$ &$3.35\pm0.07$\\ 
% 	\cline{5-7}
% 		 & & & & {$1$} &$5.98\pm0.05$ &$3.54\pm0.03$\\  \hline
% \hline
% 		 \multirow{3}{*}{$20$}&  \multirow{3}{*}{$40$}&\multirow{3}{*}{$4$}&\multirow{3}{*}{$2$}&{$0.99$} &$1.55\pm0.08$ &$3.87\pm0.12$\\	\cline{5-7}
% 		 & & & & {$0.9$} &$3.47\pm0.08$ &$4.53\pm0.11$\\ 
% 	\cline{5-7}
% 		 & & & & {$1$} &$8.41\pm0.07$ &$3.69\pm0.03$\\  \hline
% 		\end{comment}
%%%%%%%%%%%%%%%%%%%%%%%%%%%%%%%%%% Fin Comment %%%%%%%%%%%%%%%%%%%%%%

\paragraph{Results}
We collect in Table~\ref{table.comparisioncorrelation} the means of the
optimal tuning parameters $\overline{\lambda}_{min}$ and the means of the
minimal prediction errors $\overline{PE}_{min}$ for 1000 iterations and
different parameter sets. Let us first highlight the two most important
observations: first, correlations ($\rho$ large) lead to small
tuning parameters. Second, correlations do not necessarily lead to high
prediction errors. In contrast, the prediction errors are mostly smaller
for the correlated settings.\\
Let us now make some other observations. First, we find that the optimal
tuning parameters do not increase considerably when the number of
observations $n$ is increased. In
contrast, the
means of the minimal prediction errors decrease for the correlated case
as expected, whereas this is not true for the uncorrelated case.\\
Second, increasing the number of variables $p$ leads to increasing optimal tuning parameters as
expected (interestingly by factors close to $\sqrt{\frac{\log 400}{\log40}}$,
cf. Corrollary~\ref{thm.YNew}). The means of the
minimal prediction errors do, surprisingly, not increase considerably.\\
Third, as expected, increasing the sparsity $s$ does not
considerably influence
the optimal tuning parameters but leads to increasing means of the minimal prediction
errors.\\
Forth, for $\sigma=3$ both the optimal tuning parameter as well
as the mean of the minimal prediction error increase approximately by a factor
$3$. The mean of the minimal prediction errors for $\sigma=3$ and $\rho=0$
is an exception and remains unclear.\\
We finally mention that we obtained analogeous results for many other 
values of $\beta_0$ and sets of parameters. 

\paragraph{Conclusions}
The experiments illustrate the good performance of the Lasso estimator for
prediction even for highly correlated designs. Crucial is the choice of the
tuning parameters: we found that the optimal tuning parameters depend highly
on the design. This implies in particular that choosing $\lambda$
proportional to $\sqrt{n\log{p}}$ independent of the design is not favorable.

%%%%%%%%%%%%%%%%%%%%%%%%%%%%%%%%%%%%%%%%%%%%%%%%%%%%%%%%%%%%%%
%%%%%%%%%%%%%%%%%%%%%%%%%%%%%%%%%%%%%%%%%%%%%%%%%%%%%%%%%%%%%%

\section{Discussion} 
Our study suggests that correlations in the design matrix are not
problematic for Lasso prediction. However, the tuning parameter has to
be chosen suitable to the correlations. Both, the theoretical
results and the simulations strongly indicate that the larger the correlations are, the
smaller the optimal tuning parameter is. This implies in particular, that
the tuning parameter should not be chosen only as a function of the number
of observations, the number of parameters and the variance. The precise
dependence of the optimal tuning parameter on the correlations is not
known, but we expect that cross validation provides a suitable choice in
many applications.

\paragraph*{Acknowledgments}
We thank Sara van de Geer for the great support. Moreover, we 
thank Arnak Dalalyan, who read a draft of this paper carefully and gave valuable and
insightful comments. Finally, we thank the reviewers for their helpful
suggestions and comments.

\section*{Appendix}

\begin{proof}[Proof of Remark~\ref{rk.slowRateImproved}]
By the definition of the Lasso estimator, we have 
$$
\Vert Y-X \hat\beta \Vert_2^2+\lambda\Vert \hat\beta\Vert_1
\leq \Vert Y-X\beta_0 \Vert_2^2+\lambda\Vert\beta_0\Vert_1.
$$
This implies, since $Y = X\beta_0 + \sigma \epsilon$, that %absolute values
$$
\Vert X( \hat\beta - \beta_0) \Vert_2^2
\leq
2\sigma |\epsilon^T X (\hat\beta - \beta_0)| + \lambda\Vert\beta_0\Vert_1 -\lambda\Vert \hat\beta\Vert_1.
$$
Next, on the set $\mathcal{T}$, we have $2\sigma |\epsilon^T X (\hat\beta -
\beta_0)| \leq \lambda \Vert  \hat\beta - \beta_0  \Vert_1$. Additionally,
by the definition of $J_0$,%absolute values and lambda
$$
\Vert  \hat\beta - \beta_0  \Vert_1 = \Vert  (\hat\beta - \beta_0 )_{J_0} \Vert_1 + \Vert  \hat\beta_{J_0^c}  \Vert_1.
$$
Combining these two arguments and using the triangular inequality, we get
\begin{eqnarray*}
\Vert X( \hat\beta - \beta_0) \Vert_2^2
& \leq &
\lambda  \Vert  (\hat\beta - \beta_0 )_{J_0} \Vert_1 +  \lambda \Vert  \hat\beta_{J_0^c}  \Vert_1 + \lambda\Vert\beta_0\Vert_1 - \lambda\Vert \hat\beta\Vert_1
 \\
 & \leq &
 2 \lambda  \Vert  (\hat\beta - \beta_0 )_{J_0} \Vert_1.
\end{eqnarray*}
On the other hand, using the triangular inequality, we also get
$$
\Vert X (\hat\beta - \beta_0) \Vert_2^2 \leq  \lambda \Vert  \hat\beta - \beta_0  \Vert_1 + \lambda\Vert\beta_0\Vert_1 -\lambda\Vert \hat\beta\Vert_1 \leq 2 \lambda\Vert\beta_0\Vert_1.
$$
This completes the proof.

\end{proof}

\begin{proof}[Proof of Theorem~\ref{lemma.Correlation.main}]
We first show that for a fixed $x\in\mr^+_0$, the parameter space
$\{\beta\in\mr^p:\Vert\beta \Vert_1\leq M\}$ can be replaced by the $K(x)$
dimensional parameter space $\{\beta\in\mr^{K(x)}:\Vert\beta \Vert_1\leq
(1+x)M\}$. Then, we bound the stochastic term in expectation and probability and
eventually take
the infimum over $x\in\mr^+_0$ to derive
the desired inequalities.\\
As a start, we assume without loss of generality $\sigma=1$, we fix $x\in\mrpn$ and set $K:=K(x)$. Then, according to the definition of the correlation function
\eqref{eq:Kfunc}, there exist vectors $x^{(1)},...,x^{(K)}\in\sqrt nS^{n-1}$ and numbers $\{\kappa_j(m):1\leq j\leq K
\}$ for all $1\leq m\leq p$ such that $X^{(m)}=\sum_{j=1}^K\kappa_j(m)x^{(j)}$ and $\sum_{j=1}^K|\kappa_j(m)|\leq (1+x)$. Thus,
\begin{align*}
(X\beta)_i=\sum_{m=1}^pX^{(m)}_i\beta_m
=\sum_{m=1}^p\sum_{j=1}^K\kappa_j(m)x^{(j)}_i\beta_m
=\sum_{j=1}^Kx^{(j)}_i\sum_{m=1}^p\kappa_j(m)\beta_m
\end{align*}
and additionally
\begin{align*}
\sum_{j=1}^K|\sum_{m=1}^p\kappa_j(m)\beta_m|
\leq \sum_{m=1}^p |\beta_m|\sum_{j=1}^K|\kappa_j(m)|
\leq  (1+x)\Vert \beta\Vert_1.
\end{align*}
%
%%%%%%%%%%%%%%%%%%%%%%%%%%%%%%%%% Debut Comment %%%%%%%%%%%%%%%%%%%%%%
% \begin{comment} 
% \begin{align*}
% (X\beta)_i&=\sum_{j=1}^pX^{(j)}_i\beta_j\\
% &=\sum_{j\in S}X^{(j)}_i\beta_j+\sum_{m\notin S}X^{(m)}_i\beta_m\\
% &=\sum_{j\in S}X^{(j)}_i\beta_j+\sum_{m\notin S}\sum_{j\in S}\kappa_j(m)X^{(j)}_i\beta_m\\
% &=\sum_{j\in S}X^{(j)}_i\beta_j+\sum_{j\in S}X^{(j)}_i\sum_{m\notin S}\kappa_j(m)\beta_m\\
% &=\sum_{j\in S}X^{(j)}_i(\beta_j+\sum_{m\notin S}\kappa_j(m)\beta_m)
% \end{align*}
% and additionally
% \begin{align*}
% \sum_{j\in S}|\beta_j+\sum_{m\notin S}\kappa_j(m)\beta_m|
% \leq &\Vert \beta\Vert_1+\sum_{m\notin S}|\beta_m|\sum_{j\in
%   S}|\kappa_j(m)|\\
% \leq &\Vert \beta\Vert_1+\Vert \beta\Vert_1\max_{m\notin S}\sum_{j\in
%   S}|\kappa_j(m)|\\
% \leq & (1+x)\Vert \beta\Vert_1.
% \end{align*}
% \end{comment}
%%%%%%%%%%%%%%%%%%%%%%%%%%%%%%%%%%% Fin Comment %%%%%%%%%%%%%%%%%%%%%%
These two results imply
\begin{equation*}
\sup_{\Vert\beta\Vert_1\leq M}\mid \epsilon^T X\beta\mid\leq \sup_{\Vert\widetilde\beta\Vert_1\leq (1+x)M}\mid\epsilon^T \widetilde X\widetilde \beta\mid,
\end{equation*}
where $\widetilde{\beta}\in\mr^K$ and
$\widetilde{X}:=(x^{(1)},...,x^{(K)})$. That is, we can replace the $p$
dimensional parameter space by a $K$ dimensional parameter space at the
price of an additional factor $1+x$.\\
We now bound the stochastic term in expectation. First,
we obtain by Cauchy-Schwarz's Inequality
\begin{align*} 
\me\left[\sup_{\Vert\widetilde\beta\Vert_1\leq (1+x)M}\mid\epsilon^T \widetilde X\widetilde \beta\mid\right]&=
\me\left[ \sup_{\Vert\widetilde\beta\Vert_1\leq
  (1+x)M}|\sum_{i=1}^n\sum_{j=1}^K\epsilon_i\widetilde X^{(j)}_i\widetilde\beta_j|\right]\\
&\leq\left[ \me \sup_{\Vert\widetilde\beta\Vert_1\leq
  (1+x)M}\Vert \widetilde\beta\Vert_1\max_{1\leq j \leq
  K}\mid\epsilon^T\widetilde X^{(j)}\mid\right]\\
&=(1+x)M \me\left[ \max_{1\leq j \leq K}\mid\epsilon^T\widetilde X^{(j)}\mid\right].
\end{align*}
Next (cf. the proof of \cite[Lemma~3]{vdGeer11b}), we obtain for $\Psi(x):=e^{x^2}-1$
\begin{equation*}
  \me\left[ \max_{1\leq j \leq K}\mid\epsilon^T\widetilde X^{(j)}\mid\right]\leq \Psi^{-1}(K) \max_{1\leq j \leq K}\Vert\epsilon^T\widetilde X^{(j)}\Vert_\Psi,
\end{equation*}
where $\Vert\cdot\Vert_\Psi$ denotes the Orlicz norm with respect to the
function $\Psi$ (see \cite{vdGeer11b}
for a definition). Since $\frac{\epsilon^T\widetilde X^{(j)}}{\sqrt n}$ is
standard normally
distributed, we obtain $\Vert\frac{\epsilon^T\widetilde X^{(j)}}{\sqrt
  n}\Vert_\Psi=\sqrt\frac{8}{3}$ (see for
example \cite[Page 100]{vdVaart00}). Moreover, one may check that
$\Psi^{-1}(y)=\sqrt{\log(1+y)}$. Consequently,
\begin{align*} 
  \me\left[\sup_{\Vert\widetilde\beta\Vert_1\leq (1+x)M}\mid\epsilon^T
    \widetilde X\widetilde \beta\mid\right]&\leq (1+x)M
  \sqrt{\frac{8n\log(1+K)}{3}}
\end{align*}
One can then derive the second assertion of the
theorem by taking the infimum over $x\in\mrpn$.\\
\noindent As a next step, we deduce similarly as above (compare also to \cite{Bickel09})
\begin{align*}
\mpr\left(\sup_{\Vert\beta\Vert_1\leq M}2\mid\epsilon^T X\beta\mid\geq \lambda M\right)\leq&\mpr\left(\sup_{\Vert\widetilde\beta\Vert_1\leq 1}2\mid\epsilon^T \widetilde X\widetilde \beta\mid\geq \frac{\lambda}{1+x}\right) \\
\leq& K\max_{1\leq j \leq K}\mpr\left(\left| \epsilon^T \frac{\widetilde X^{(j)}}{\sqrt n}\right| \geq \frac{\lambda }{2(1+x)\sqrt n}\right)\\
=&K\mpr\left(|\gamma |\geq  \frac{\lambda}{2(1+x) \sqrt n}\right),
\end{align*}
where $\gamma$ is a standard normally distributed random variable. 
Setting
$\lambda:=2(1+x)\sqrt{ 2n\log(2K/\kappa)}$, we obtain
\begin{align*}
\mpr\left(\sup_{\Vert\beta\Vert_1\leq M}2\mid\epsilon^T X\beta\mid\geq \lambda M\right)
&\leq 2K\exp{\left(\text{-}\frac{\lambda^2}{8(1+x)^2n}\right)}\\
&=2K\exp{\left(\text{-}\log(2K/\kappa)\right)}\\
&=\kappa.
\end{align*}
The first assertion can finally be  derived taking the infimum
over $x\in\mrpn$ and applying monotonous convergence.
\end{proof}

%%%%%%%%%%%%%%%%%%%%%%%%%%%%%%%
%%%%%%%%%%%%%%%%%%%%%%%%%%%%%%%

\begin{proposition}\label{lemma.coveringnumbers}
  It holds for all $0<\delta\leq 1$
  \begin{equation*}
 \left(\frac{1}{\delta}\right)^n\leq  N(\delta, \{x\in\mr^n:\|x\|_2\leq 1\},\|\cdot\|_2)\leq \left(\frac{3}{\delta}\right)^n.  
  \end{equation*}

\end{proposition}
\begin{proof}
  If a collection of balls covers the set
  $\{x\in\mr^n:\|x\|_2\leq 1\}$, the volume covered by these balls is at least as
  large as the volume of $\{x\in\mr^n:\|x\|_2\leq 1\}$. Thus,
\begin{equation*}
  N(\delta, \{x\in\mr^n:\|x\|_2\leq 1\},\|\cdot\|_2)\cdot\delta^n\geq 1,
\end{equation*}
and the left inequality follows. The right inequality can be deduced similarly.
\end{proof}

\begin{proof}[Proof Sketch for Example~\ref{example.Correlation.vhdd}]
For weakly correlated designs and for $p$ very large, it holds that 
\begin{align*}
& \log\left(1+N(\sqrt
  n\delta,\operatorname{sconv}\{X^{(1)},...,X^{(p)}\},\Vert \cdot
  \Vert_2)\right)\\
\approx & \log\left(N(\delta,\{x\in\mr^n:\|x\|_2\leq 1\},\Vert \cdot
  \Vert_2)\right).
  \end{align*}
We can then apply Proposition~\ref{lemma.coveringnumbers} to deduce 
\begin{align*}
 \log\left(1+N(\sqrt
  n\delta,\operatorname{sconv}\{X^{(1)},...,X^{(p)}\},\Vert \cdot
  \Vert_2)\right)\approx n \log(1/\delta).
  \end{align*}
Hence, Lemma~\ref{lemma:SJ1} requires that $A$ fulfills for all $0<\delta
\leq 1$ (in fact, it is sufficient for $\frac 1 {\sqrt n}\lesssim\delta
\leq 1$, see the proofs in \cite{Lederer10} for this refinement)
\begin{align*}
 &\left(\frac{A}{\delta}\right)^{2\alpha}\approx n \log(1/\delta)
  \end{align*}
and thus
\begin{align*}
&A^{\alpha}\approx \delta^\alpha\sqrt{n\log(1/\delta)}.
  \end{align*}
We can maximize the right hand side over $0<\delta\leq 1$ and plug the result into \cite[Corollary 5.2]{vdGeer11} to deduce the result.   
\end{proof}

\begin{proof}[Proof for Example~\ref{ex.asym}]
The realizations of the random vectors $X^{(1)}+\nu N$ are clustered with high probability and
can thus be described (in the sense of the definition of $K_\kappa$) by much less than $p$ vectors.\\
To see this, consider $X^{(1)}$
and a rotation $R\in \operatorname{SO(n)}$ such that $RX^{(1)}=(\sqrt n,
  0,\dots,0)^T$. Moreover, let $N$ be a standard normally
  distributed random
  vector in $\mr^n$, $n\geq 2$. The measure corresponding 
  to the random vector $N$ is denoted by $\mpr$. We then obtain with the triangle inequality and the condition on $\nu$ 
\begin{align*}
\mpr\left(\left|1-\frac{\sqrt n}{\|RX^{(1)}+\nu R N\|_2}\right|\geq{1}\right)
= & \mpr\left(\frac{\|RX^{(1)}+\nu R N \|_2}{\sqrt n}\leq \frac 1 2\right)\\
\leq & \mpr\left(\frac{\|X^{(1)}\|_2-\nu\| N\|_2}{\sqrt n}\leq \frac 1 2\right)\\
= & \mpr\left(\frac{\nu\| N\|_2}{\sqrt n}\geq \frac 1 2\right)\\
\leq & \mpr\left(\| N\|_2\geq \frac{\sqrt n}{2\nu}\right)\\
\leq & \mpr\left(\| N\|_2\geq \sqrt{2n}\right).
\end{align*}
Now, we can bound the $l_1$ distance of the vector $\frac{\sqrt n(RX^{(1)}+\nu R N)}{\|RX^{(1)}+\nu R N\|_2}$ generated
according to our setting to the vector $RX^{(1)}$: 
\begin{align*}
 & \mpr\left(\left\|RX^{(1)}-\frac{\sqrt n(RX^{(1)}+\nu R
     N)}{\|RX^{(1)}+\nu R N\|_2}\right\|_1\geq2{\sqrt n}\right)\\
\leq& \mpr\left( \left|1-\frac{\sqrt
      n}{\|RX^{(1)}+\nu R N\|_2}\right|\|RX^{(1)}\|_1\geq{\sqrt n}\right)+
\mpr\left(\|{\nu R N}\|_1\geq{\|RX^{(1)}+\nu R N\|_2}\right)\\
\leq& \mpr\left(\left|1-\frac{\sqrt  n}{\|RX^{(1)}+\nu
    R N\|_2}\right|\geq 1\right)+\mpr\left(\sqrt n\nu\|
     N\|_2\geq \sqrt n-{\nu \| N\|_2}\right)\\
\leq& \mpr\left(\| N\|_2\geq \sqrt{2n}\right)+\mpr\left(\|
     N\|_2\geq\frac{1}{2\nu }\right)\\
\leq& 2\mpr\left(\|
     N\|_2\geq\sqrt{2n}\right).
\end{align*}
 It can be derived easily from 
standard results (see for example \cite[Page 254]{Buhlmann11}) that
\begin{align*}
  & \mpr\left(\| N\|_2\geq \sqrt {2n}\right)\leq\exp\left(\text-\frac{n}{8}\right).
\end{align*}
Consequently,
\begin{equation}\label{eq.New2}
  \mpr\left(\left\|RX^{(1)}-\frac{\sqrt n(RX^{(1)}+\nu R
      N)}{\|RX^{(1)}+\nu R N\|_2}\right\|_1\geq 2\sqrt n\right)\leq 2\exp\left(\text-\frac{n}{8}\right).
\end{equation}
Finally, define for any $0\leq z\leq 2\sqrt n$ the set 
\begin{align*}
 \mathcal C_z:=\{(\sqrt n,z,
  0,\dots,0)^T&,(\sqrt n,\text - z,
  0,\dots,0)^T, (\sqrt n,0,z,
  0,\dots,0)^T,\\
&\hspace{2.5cm}(\sqrt n,0,\text -z,
  0,\dots,0)^T, \dots\}\subset\mr^n 
\end{align*}
These vectors will play the role of the vectors $x^{(1)},...,x^{(l)}$ in
Definition~\eqref{eq:Kfunc}. It holds that 
\begin{equation}\label{eq.New3}
\operatorname{card}(\mathcal C_z)=2(n-1)
\end{equation}
and
\begin{equation}\label{eq.New4}
\mathcal C_z\subset \sqrt{n+z^2}~ S^{n-1}.
\end{equation}
Moreover,
\begin{equation}\label{eq.New5}
\{x\in\mr^n:\|x-RX^{(1)}\|_1\leq z\}\cap\sqrt n S^{n-1}\subset\operatorname{sconv}(\mathcal C_z).
\end{equation}
Inequality~\eqref{eq.New2} and  Inclusion~\eqref{eq.New5} imply that
with probability at least $1-2\exp\left(\text-\frac{n}{8}\right)$
\begin{equation*}
\frac{\sqrt n(X^{(1)}+\nu 
      N)}{\|RX^{(1)}+\nu R N\|_2}
\in\operatorname{sconv}(R^T\mathcal C_{{2\sqrt n}}).
\end{equation*}
Thus, using Equality~\eqref{eq.New3} and Inclusion~\eqref{eq.New4},  we
have $K(\sqrt{5}-1) \leq 2(n-1)$ with probability at least $1-2(p-1)\exp\left(\text-\frac{n}{8}\right)$. 
\end{proof}

\bibliographystyle{alpha}      % basic style, author-year citations
\renewcommand{\refname}{}
\bibliography{Literature}

\end{document}